\documentclass[a4paper,11pt]{amsart}
\usepackage{mathrsfs}
\usepackage{appendix}
\usepackage{amssymb}
\usepackage{fancyhdr}
\usepackage{charter} 
\usepackage{typearea}
\usepackage{color}
\usepackage{amsmath,amstext,amsthm,amscd}
\usepackage{hyperref}

\usepackage[a4paper,top=2.5cm,bottom=2.5cm,left=2cm,right=2cm]{geometry}

\makeatletter
\def\@setcopyright{}
\def\serieslogo@{}
\makeatother
\newcommand{\boldsym}[1]{\boldsymbol{#1}}
\newcommand\bn{\boldsym{n}}

\newcommand{\dbar}{\partial}
\newcommand{\ddbar}{\overline\partial}

\newcommand{\ol}{\overline}

\newcommand{\norm}[1]{\left\Vert#1\right\Vert}

\newcommand{\set}[1]{\left\{#1\right\}}

\DeclareMathOperator{\Ker}{Ker}

\newcommand{\cali}[1]{\mathscr{#1}}
\newcommand{\cH}{\cali{H}}

\newcommand{\mT}{\mathcal{T}}
\newcommand{\R}{\mathbb{R}}

\newcommand{\ov}{\overline}
\newcommand{\til}[1]{\widetilde{#1}}
\newcommand{\wi}{\widetilde}
\def\Im{{\rm Im}}

\DeclareMathOperator{\supp}{supp}
\DeclareMathOperator{\Dom}{Dom}

\def\cL{\mathscr{L}}
\theoremstyle{plain}
\newtheorem{theorem}{Theorem}[section]
\newtheorem{lemma}[theorem]{Lemma}
\newtheorem{corollary}[theorem]{Corollary}
\newtheorem{proposition}[theorem]{Proposition}

\newtheorem{definition}[theorem]{Definition}

\newtheorem{remark}[theorem]{Remark}

\numberwithin{equation}{section}

\begin{document}
	
	\title[Asymptotics of Spectral Functions of Lower Energy Forms on Weakly 1-Complete Manifolds]
	{Asymptotics of Spectral Functions of Lower Energy Forms on Weakly 1-Complete Manifolds}

\author[]{Xiquan Peng}
\address{School of Mathematical Sciences, Fudan University, Shanghai 200433, China}
\thanks{}
\email{23110180028@m.fudan.edu.cn}

\author[]{Guokuan Shao}
\address{School of Mathematics (Zhuhai), Sun Yat-sen University, Zhuhai 519082, Guangdong, China}
\thanks{Guokuan Shao was supported by NSFC Grant No.12001549 and Natural Science Foundation of Guangdong Province Grant No.  2023A1515030017}
\email{shaogk@mail.sysu.edu.cn}

\author[]{Wenxuan Wang}
\address{School of Mathematics (Zhuhai), Sun Yat-sen University, Zhuhai 519082, Guangdong, China}  
\email{wangwx67@mail2.sysu.edu.cn}

\date{}

\begin{abstract}
Let $M$ be a weakly $1$-complete manifold. We show that the optimal fundamental estimate holds true on $M$ with mild conditions, then we establish the weak Morse inequalities for lower energy forms on $M$. We also study the case for $q$-convex manifolds.
\end{abstract}

\maketitle
\section{Introduction}

The study of holomorphic Morse inequalities dates back to Demailly's classical work \cite{Dem:85}. It was motivated by Grauert-Riemenschneider conjecture \cite{Siu:84}, which stated that if a holomorphic line bundle $L$ is semi-positive and positive at least one point, then $L$ is big. Inspired by Witten's analytic proof \cite{Wit82} of the classical Morse inequality, Demailly replaced the Morse function by the Hermitian metric of $L$ and the Hessian of the Morse function by the curvature form of $L$. Since then the holomorphic Morse inequalities have been explored intensively during the last four decades. Bismut \cite{Bis87} provided a heat kernel proof involving probability theory. Berman \cite{Be04} investigated asymptotics of Bergman kernel functions and then deduced a local version of holomorphic Morse inequalities by scaling technique and standard arguments from functional analysis. Ma-Marinescu \cite{MM07} gave another heat kernel proof inspired by Bismut-Lebeau\cite{BisLe91}. Bouche \cite{Bou:89} considered the $q$-convex case. Marinescu \cite{M:92} established a result on weakly $1$-complete manifold. Recently, Hsiao and his collaborators \cite{HM12, HsiaoLi:16,HHLS20,HS20} obtained variant versions in Cauchy-Riemann manifolds and complex manifolds with boundary. Besides solving the Grauert-Riemenschneider conjecture, the holomorphic Morse inequalities also give some key evidence towards the Green-Griffiths-Lang conjecture\cite{Dem11}.

The formulas of the holomorphic Morse inequalities on non-compact manifolds \cite[Theorem1.2-Theorem1.6]{LSW} can not be established globally in general. The key ingredients to deal with such various settings on non-compact manifolds rely on analysis about positivities of bundles or positivities of boundaries of manifolds. Li-Shao-Wang \cite{LSW} introduced the concept of optimal fundamental estimate and gave a unified proof of the weak holomorphic Morse inequalities on several non-compact settings. They also studied asymototics of spectral function of lower energy forms and obtained weak holomorphic Morse inequalities for lower energy forms on complete Hermitian manifolds. 

In this paper, we mainly focus on weakly $1$-complete manifolds. The motivation is to establish a version of holomorphic Morse inequalities on non-compact setting. Our results could provide a tool to study embedding problems of non-compact manifolds. We will show that the optimal fundamental estimate holds on weakly $1$-complete manifolds and establish a version of Morse inequalities for lower energy forms on such manifolds.

The following is our main result. An asymptotics of dimensions of the lower energy form spaces for weakly $1$-complete manifolds with some specifical line bundle is established. For the notations and definitions, we refer the readers to Section 2 and Section 3.
\begin{theorem}\label{maint}
	Let $M$ be a weakly $1$-complete manifold of dimension $n$. Let $(L,h^L)$ and $(E,h^E)$ be 
 holomorphic Hermitian line bundles on $M$. Assume $K$ is a compact subset and $(L,h^L)$ is Griffiths $q$-positive on $X\setminus K$ with $q\geq 1$. Then there exists a sequence of relatively compact sets $\{M_i\}_{i=1}^{\infty}$ of $M$, such that $K\subset M_1$, $\cup_{i} M_i=M$ and $M_i\Subset M_{i+1}$ , and a sequence of Hermitian metrics $\{\omega_i\}_{i=1}^{\infty}$ on $M$. As $k\rightarrow \infty$, 
    for any $j\geq q$ and $i\in \mathbb{N}^{+}$ , we have
	\begin{equation}
		\dim\mathcal{E}^j(k^{-N_{0}},\square^{k,E}_{i}) \leq\frac{k^{n}}{n!} \int_{K(j)}(-1)^j c_1(L,h^L)^n+o(k^{n}),
	\end{equation}
	where $\square^{k,E}_{i}$ is the Gaffney extension of $\ddbar^{k,E} \ddbar^{k,E*}_{i}+\ddbar^{k,E*}_{i} \ddbar^{k,E}$ and  $\ddbar^{k,E*}_{i}$ is Hilbert adjoint of $\ddbar^{k,E}$ under the norm $ \|\cdot\|_{i}$ on $M_i$, which is induced by $\omega_i, h^L$ and $h^E$.
\end{theorem}

We also obtain a similar result on $q$-convex manifolds.

\begin{corollary}\label{mainc}
	Let $M$ be a $q$-convex manifold of dimension $n$ and $1\leq q \leq n$. Let $(L,h^L)$ and $(E,h^E)$ be holomorphic Hermitian line bundles on $X$. Suppose the curvature $\Theta_L$ of $L$  has at least $n-s+1$ non-negative eigenvalues on $M\setminus K$ for a compact subset $K$ with $1\leq s\leq n$.
	Then for each $s+q-1\leq j\leq n$, we deduce the estimate for the dimension of the $j$-th lower energy form spaces $\mathcal{E}^j(k^{-N_{0}},\square^{k,E}_{c,\chi})$,
	\begin{equation}
		\dim\mathcal{E}^j(k^{-N_{0}},\square^{k,E}_{c,\chi}) \leq\frac{k^{n}}{n!} \int_{K(j)}(-1)^j c_1(L,h^L)^n+o(k^{n}),
	\end{equation}
	where $\square^{k,E}_{c,\chi}$ is the Gaffney extension of $\ddbar^{k,E} \ddbar^{k,E*}_{c,\chi}+\ddbar^{k,E*}_{c,\chi} \ddbar^{k,E}$. 
\end{corollary}

The paper is organized as follows: we present basic notions and results needed in our proofs in Section 2. In Section 3, we give a proof of asymptotic estimate of dimension of the lower energy form spaces from the optimal fundamental estimate. In Section 4, we are devoted to proving our main theorem.

\section{Preliminaries} \label{Sec_pre}

\subsection{Complex manifolds}
Let $(M, \omega)$ be a Hermitian manifold of dimension $n$ with Hermitian metric $\langle\cdot,\cdot\rangle$, where $\omega$ is the fundamental form induced by the Hermitian metric. Consider arbitrary Hermitian  bundle $(E,h^E)$ on $M$. The Hermitian metrics on $M$ and $E$ induce the Hermitian metric on $T^{*p,q}M \otimes E$ for $p,q=0,1,\ldots,n$. We shall also denote  such metric by $\langle\cdot,\cdot\rangle$ and denote its pointwise norm by $|\cdot|$.
We denote the space of smooth $(p,q)$-forms over $M$ with values in $E$  by $\Omega^{p,q}(M, E)$. 
Let $\Omega^{p,q}_0(M, E)$ be the subspace of $\Omega^{p,q}(M, E)$ whose elements have compact support.
The metric $\langle\cdot,\cdot\rangle$ and the volume form $dv_M := \frac{\omega^n}{n!}$ induce a Hermitian inner product $(\cdot,\cdot)$ on $\Omega^{p,q}_0(M, E)$ by defining
\begin{equation}
	(s_{1},s_{2}):=\int_{M}\langle s_{1}(x), s_{2}(x)\rangle dv_{M}(x).
\end{equation}
Write $\norm{\cdot}^2 := (\cdot, \cdot)$.
We let $L^2_{p,q}(M, E)$ denote the completion of $\Omega^{p,q}_0(M, E)$ with respect to $(\cdot, \cdot)$.

It is well known that the Dolbeault operator $\ddbar^{E}:\Omega_0^{p,q} (M, E)\rightarrow \Omega_0^{p,q+1} (M, E)$ satisfies $(\ddbar^{E})^2 = 0$. By abuse of notation, let $\ddbar^{E}:L^2_{p,q}(M,E)\rightarrow L^2_{p,q+1}(M,E)$ denote the associated maximal $L^2$-extension, which is a closed, densely defined operator that satisfies $(\ddbar^{E})^2 = 0$. Let $\ddbar^{E*}$ denote the Hilbert space adjoint of $\ddbar^{E}$, then the Gaffney extension of Kodaira Laplacian is given by 
\begin{eqnarray}\label{eq40}\nonumber
	\Dom(\square^{E})&=&\{s\in \Dom(\ddbar^{E})\cap \Dom(\ddbar^{E*}):
	\ddbar^{E}s\in \Dom(\ddbar^{E*}),~\ddbar^{E*}s\in \Dom(\ddbar^{E}) \}, \\
	\square^{E}s&=&\ddbar^{E} \ddbar^{E*}s+\ddbar^{E*} \ddbar^{E}s \quad \mbox{for}~s\in \Dom(\square^{E}),
\end{eqnarray}
Which is a positive, self-adjoint operator  \cite[Proposition 3.1.2]{MM07}.
We define the space of harmonic $(p,q)$-forms with values in $E$ as
\begin{equation}\label{eq45}
	\cH^{p,q}(M,E):=\Ker(\square^{E})\cap  L^2_{p,q}(M,E) =\{s\in \Dom(\square^{E})\cap L^2_{p,q}(M, E): \square^{E}s=0 \}.
\end{equation}
Let the $q$-th $L^2$-Dolbeault cohomology be defined as quotient
\begin{equation}\label{eq46}
	H^{0,q}_{(2)}(M,E):=(\Ker(\ddbar^{E})\cap  L^2_{0,q}(M,E) )\big/  (\Im( \ddbar^{E}) \cap L^2_{0,q}(M,E)),
\end{equation}
while the $q$-th reduced $L^2$-Dolbeault cohomology is defined as quotient
\begin{equation}\label{eq46}
	\overline{H}^{0,q}_{(2)}(M,E):=(\Ker(\ddbar^{E})\cap  L^2_{0,q}(M,E)) \big/\overline{(\Im( \ddbar^{E}) \cap L^2_{0,q}(M,E))},
\end{equation}
where $\overline{(\,\cdot\,)}$ denotes the closure of the space with respect to $(\cdot, \cdot)$.

Besides, the Dolbeault isomorphism $H^q(M,E) \cong H^{0,q}(M,E)$ is well known, where $H^q(M,E)$ is the $q$-th sheaf cohomology of the sheaf $\mathscr{O}_M(E)$ of holomorphic section of $E$ over $M$, and $H^{0,q}(M,E)$ is the $q$-th Dolbeault cohomology. 

Suppose $(L,h^L )$ is a Hermitian line bundle on $M$, the curvature of Chern connection on $(L, h^L)$ will be denoted by $\Theta_L=\ddbar\dbar \log|s|^2_{h^{L}}$ for any local holomorphic trivializing section $s$, and the associated Chern curvature form $c_{1}(L,h^{L}):=\frac{\sqrt{-1}}{2\pi}\Theta_L$ is a real $(1,1)$-form on $M$. 
We identify the two-form $\Theta_L$ at every $x\in M$ with a linear endomorphism $Q_L$ given by
\begin{equation}\label{Ql}
    \Theta_L (\alpha, \ol{\beta}) = \langle Q_L \alpha, \ol{\beta}\rangle,\; \forall \, \alpha, \beta\in T^{1,0}_xM. 
\end{equation}

We can check that the numbers of positive, negative and zero eigenvalues of $Q_L$ are independent of the choice of the Hermitian metric $\omega$ of $M$.

\subsection{Some estimates}\label{subsec_ofe}
The following estimate should be viewed as Morrey-Kohn-H\"{o}rmander estimate that plays an important role in the solution of $\ddbar$-Neumann problem.
\begin{theorem}[{\cite[Corollary 1.4.22]{MM07}}]\label{BKNwithbd}
	Let $M$ be a smooth relatively compact domain in a Hermitian manifold $(X,\omega)$. Let $\rho\in C^\infty(X)$ such that $M=\{x\in X: \rho(x)< 0\}$ and $|d\rho|=1$ on the boundary $bM$. Let $(L,h^L)$ be a holomorphic Hermitian line bundle on $X$. Then for any $s\in B^{0,q}(M,L)$, $0\leq q\leq n$,
	\begin{equation}
		\begin{split}
			\frac{3}{2}\left(||\ddbar^L s||^2+||\ddbar^{L*}s||^2\right)
			&\geq \frac{1}{2}||(\nabla^{\til{L}})^{1,0*}\til{s}||^2+\left(\Theta_{L}\otimes K^{*}_{X}(w_j,\ov{w}_k)\ov{w}^k\wedge i_{\ov{w}_j} s,s\right)\\
			&+\int_{bM}\cL_{\rho}(s,s)dv_{bM}-\frac{1}{2}\left(||\mT^*\til{s}||^2+||\ov{\mT}\til{s}||^2+||\ov{\mT}^* \til{s}||^2\right),
		\end{split}
	\end{equation}
	where $\wi{L}:=L\otimes K^*_X$, $\wi s:=(w^1\wedge\cdots\wedge w^n\wedge s)\otimes (w_1\wedge\cdots\wedge w_n)$, and $\nabla^{\wi L}$ is the Chern connection; $\{ w_j \}_{j=1}^n$ is a local orthonormal basis of $T^{1,0}X$ with dual basis $\{ w^j\}_{j=1}^n$ of $T^{1,0*}X$; $\cL_{\rho}(\cdot,\cdot):=(\dbar\ddbar \rho)(w_j,\ov w_k)\langle \ov w^k\wedge i_{\ov w_j} \cdot,\cdot\rangle_{\Lambda^{0,q}\otimes L}$ is the Levi form of $bM$; $\mT$ is the Hermitian torsion operator and $B^{0,q}(M,L)$ is defined in Section 4 .
\end{theorem}

Let $0\leq q\leq n$. We say the \textbf{fundamental estimate} about $(0,q)$-forms with values in $E$ holds,  provided there are compact subset $K\subset M$ and  $C>0$, such that for any $s\in \Dom(\ddbar^E)\cap\Dom(\ddbar^{E*})\cap L^2_{0,q}(M,E)$,
\begin{equation}\label{fe1}
	\|s\|^2\leq C\left(\|\ddbar^E s\|^2+\|\ddbar^{E*} s\|^2+\int_K|s|^2dv_M\right).
\end{equation}
The above fundamental estimate yields $\ov H_{(2)}^{0,q}(M,E)\cong H^{0,q}_{(2)}(M,E)$ \cite[Theorem 3.1.8]{MM07}. 

From now on, we suppose $\operatorname{rank}(E) = 1$. Denoting $\ddbar^{L^k\otimes E}$ by $\ddbar^{k,E}$, and its Hilbert space adjoint by $\ddbar^{k,E*}$, we say the \textbf{optimal fundamental estimate} about $(0,q)$-forms with values in $L^k\otimes E$ holds, provided there are compact subset $K\subset M$, and $ C>0$, such that for any $k \gg 1$ and $ s\in \Dom(\ddbar^{k,E})\cap\Dom(\ddbar^{k,E*})\cap L^2_{0,q}(M,L^k\otimes E)$,
\begin{equation}
	\left(1- \frac{C}{k} \right)\|s\|^2\leq \frac{C}{k} \left(\|\ddbar^{k,E} s\|^2+\|\ddbar^{k,E*} s\|^2\right) +\int_K|s|^2dv_M.
\end{equation}
The above optimal fundamental estimate was introduced in {\cite{LSW}}; it is clear that optimal fundamental estimate can deduce fundamental estimate for forms with values in $L^k\otimes E$.

\subsection{Local asymptotics of spectral functions of lower energy forms}
For $q=1,\ldots,n$,  let $\square^{k,E}$ be the Gaffney extension of Kodaira Laplacian  $\ddbar^{k,E} \ddbar^{k,E*}+\ddbar^{k,E*} \ddbar^{k,E}$. It is known that the spectrums of $\square^{k,E}$ lie in $[0, +\infty)$ by \cite[Proposition 3.1.2]{MM07}. Let $\Omega\subset [0, +\infty)$ be a Borel set, and denote by $E^q(\Omega) : L^2_{0,q}(M, L^k\otimes E) \rightarrow \operatorname{Im}E^q (\Omega)$ the spectral projection of $\square^{k,E}$ with respect to $\Omega$. We set $\mathcal{E}^q(\lambda,\square^{k,E}) :=\operatorname{Im}E^q\left([0, \lambda] \right)$, and call it the spectral space of $\square^{k,E}$ corresponding to energy lower than $\lambda$.  Take an orthonormal basis $\{s_j\}_{j\geq1}$ of $\mathcal{E}^q(\lambda,\square^{k,E})$ and set $B^q_{\leq \lambda}(x) := \sum_{j}|s_j(x)|^2_{L^k\otimes E}$ to be the spectral function, which is independent of the choice of orthonormal basis.

Let $c_1({L}, h^{{L}})$ be the first Chern class of $(L, h^L)$, and $M(q)$ be the subset of $M$ consisting of points on which the curvature of $(L,h^L)$ is non-degenerate and has exactly $q$ negative eigenvalues. The following two theorems can be deduced by \cite[Corollary 1.4]{HM:14}:
\begin{theorem}\label{sfae}
	Let $N_0 \geq 2n+1$, then the spectral function $B^q_{\leq k^{-N_0}}(x)$ has an asymptotic equality
	\begin{equation}
		\begin{split}
			\limsup_{k\rightarrow \infty}k^{-n}B^q_{\leq k^{-N_0}}(x)= (-1)^q 1_{{M}(q)}\frac{c_1({L},h^{{L}})^n}{{\omega}^n}({x}),
		\end{split}
	\end{equation}
	where $1_{{M}(q)}$ is the characteristic function of $M(q)$.
\end{theorem}

\begin{theorem}\label{sfbd}
	Let $K \subset M$ be a compact subset. Then there exist $C>0$, and $k_0 >0$, such that for any $ k> k_0, x \in K$, we deduce that
	\begin{equation}
		\begin{split}
			k^{-n}B^q_{\leq k^{-N_0}}(x)\leq C .
		\end{split}
	\end{equation}
\end{theorem}

We give some definitions at the end of this section.
\begin{definition}
	A complex manifold $M$ is said to be weakly $1$-complete{ \cite{Nak:70}}, if there is a plurisubharmonic function $\varphi \in C^{\infty}(M, \mathbb{R})$, such that $M_c:=\{x \in M : \varphi(x) <c \} \Subset M$ for any $c \in\mathbb{R}$. Such $\varphi$ is called an exhaustion function of $M$.
\end{definition}

\begin{definition}
	A Hermitian line bundle $(L,h^L)$ on a complex manifold $M$ is said to be Griffiths $q$-positive at $x \in M$, if the curvature form $\Theta_L$ has at least $n-q+1$ positive eigenvalues at $x$, where $n =\operatorname{dim}_{\mathbb{C}} M$, $1 \leq q \leq n$. 
\end{definition}

\begin{definition}
	A complex manifold $M$ of dimension $n$ is called $q$-convex{\cite{AG:62}}, if there is a smooth function $\varrho\in C^\infty(X,\R)$ such that the sublevel set $M_c=\{ \varrho<c\}\Subset M$ for all $c\in \R$ and the complex Hessian $\dbar\ddbar\varrho$ has $n-q+1$ positive eigenvalues outside a compact subset $K\subset M$. We call $\varrho$ an exhaustion function and $K$ exceptional set.
\end{definition}

\section{Asymptotics of spectral functions of lower energy forms}\label{Sec_l2wmi}

When the optimal fundamental estimate is satisfied, {\cite[Proposition 4.2]{LSW}} gives an asymptotic estimate of the dimension of the lower energy form spaces $\dim \mathcal{E}^q(k^{-N_{0}},\square^{k,E}), N_{0}\ge 2n+1$ as follows

\begin{theorem}{\cite[Proposition 4.2]{LSW}}\label{main1}
	Let $0\leq q\leq n$. For $N_{0}\ge 2n+1$, the estimate of dimension of the lower energy form spaces 
	\begin{equation}
		\dim\mathcal{E}^q(k^{-N_{0}},\square^{k,E}) \leq\frac{k^{n}}{n!} \int_{K(q)}(-1)^q c_1(L,h^L)^n+o(k^{n})
	\end{equation}
	holds, provided there are a compact subset $K\subset M$ and $ C>0$, such that for any $k \gg 1$ and $ s\in \Dom(\ddbar^{k,E})\cap\Dom(\ddbar^{k,E*})\cap L^2_{0,q}(M,L^k\otimes E)$ 
 the inequality as follows holds,
	\begin{equation}
		\left(1- \frac{C}{k} \right)\|s\|^2\leq \frac{C}{k} \left(\|\ddbar^{k,E} s\|^2+\|\ddbar^{k,E*} 
		s\|^2\right) +\int_K|s|^2dv_M.
	\end{equation}
\end{theorem}
Following the proof of {\cite{LSW}}, we give a more detailed proof of this theorem.
\begin{proof}
	
	By the spectral decomposition theorem, using the optimal fundamental estimate, for any $s\in \mathcal{E}^q(k^{-N_0},\square^{k,E})\subset \Dom(\square^{k,E})\cap L^2_{0,q}(M,L^k\otimes E)$,
	\begin{equation}
		\begin{split}
			&(1-\frac{C}{k})||s||^2\leq \frac{C}{k}(||\ddbar^{k,E}s||^2+||\ddbar^{E*}_{k}s||^2)+\int_{K} |s|^2 dv_M\\
			=&\frac{C}{k}(\square^{k,E}s,s)+\int_{K} |s|^2 dv_M
			=\frac{C}{k}\left(\left(\int_{\mathbb{R}}\lambda dE_{\lambda}^{q}\right)s,s\right)+\int_{K} |s|^2 dv_M\\
			=&\frac{C}{k}\int_{\mathbb{R}}\lambda d(E_{\lambda}^{q}s,s)+\int_{K} |s|^2  dv_M
			=\frac{C}{k}\int_{[0,k^{-N_{0}}]}\lambda d(E_{\lambda}^{q}s,s)+\int_{K} |s|^2  dv_M\\
			\le &\frac{C}{k} k^{-N_{0}}\int_{\mathbb{R}}d(E_{\lambda}^{q}s,s)+\int_{K} |s|^2  dv_M
			=Ck^{-N_{0}-1}\left(\left(\int_{\mathbb{R}}1dE_{\lambda}^{q}\right)s,s\right)+\int_{K} |s|^2  dv_M\\
			=&Ck^{-N_{0}-1}\|s\|^2+\int_{K} |s|^2 dv_M,	
		\end{split}		
	\end{equation}
	where $E_{\lambda}^{q}:=E^{q}(-\infty, \lambda)$, is the spectral measure of $\square^{k,E}$.
	Thus, it follows that $\|s\|^2\leq c_k\int_{K} |s|^2 dv_X$, where $c_k:=\frac{k^{N_{0}+1}}{k^{N_{0}+1}-Ck^{N_{0}}-C}$, and $c_k \to 1$, when $k\to \infty$.
	By Fatou's lemma, Theorem\ref{sfbd}, H\"{o}lder's inequality and Theorem\ref{sfae}, we get:
	\begin{equation}
		\begin{split}
			&\limsup_{k\rightarrow \infty}\left(k^{-n}\dim \mathcal{E}^q(k^{-N_0},\square^{k,E})\right)\\
			\leq&\limsup_{k\rightarrow \infty}\left(k^{-n}c_k \int_{K}B_{\leq k^{-N_0}}^q(x)dv_X(x)\right)\\
			\leq&\left(\limsup_{k\rightarrow \infty} c_k\right)\left(\limsup_{k\rightarrow \infty} \int_{K}k^{-n}B_{\leq k^{-N_0}}^q(x)dv_X(x)\right)\\
			\leq&\int_{K}\limsup_{k\rightarrow \infty}k^{-n} B_{\leq k^{-N_0}}^q(x)dv_M(x)
			\leq \int_{K(q)}(-1)^q \frac{c_1({L},h^{{L}})^n}{n!}.
		\end{split}
	\end{equation}
\end{proof}

Using the canonical isomorphism of the weakly Hodge decomposition{\cite[(3.1.21)]{MM07}} $\cH^{0,q}(M,E) \cong \overline{H}^{0,q}_{(2)}(M,E)$, and the fact that we have $\overline{H}^{0,q}_{(2)}(M,E)\cong H^{0,q}_{(2)}(M,E)$, when the fundamental estimate holds{\cite[Theorem 3.1.8]{MM07}, we have $\dim H^{0,q}_{(2)}(M,L^{k}\otimes E)= \dim \cH^{0,q}(M,L^{k}\otimes E)\le \dim\mathcal{E}^q(k^{-N_{0}},\square^{k,E})$.} This allows us to get the weak Morse inequality by Theorem \ref{main1}.

\begin{corollary}
	Let $0\leq q\leq n$. For $N_{0}\ge 2n+1$, under the same assumption as Theorem \ref{main1}, we deduce the estimate of dimension of the $L^2$-Dolbeault cohomology group
	\begin{equation}
		\dim H^{0,q}_{(2)}(M,L^{k}\otimes E) \leq\frac{k^{n}}{n!} \int_{K(q)}(-1)^q c_1(L,h^L)^n+o(k^{n}).
	\end{equation}
	
\end{corollary}

For a $q$-convex manifold $M$ with  modified hermitian line bundle $L_{\chi}^k :=(L^k ,h_{\chi}^{L^k})$ and $(E,h^E)$, it is well known that there is $c\in \mathbb{R}$, such that  the sublevel set $M_c=\{ \varrho<c\}\Subset M$ satisfies optimal fundamental estimate {\cite[proposition 3.4]{LSW}}, where $h_{\chi}^{L^k}:= h^{L^k}e^{-k\chi(\varrho)}$ and $\chi \in C^{\infty}(\mathbb{R})$ is some suitable rapidly convex-increasing function. That is
\begin{equation}\label{covoe}
	\left(1- \frac{C}{k} \right)\|s\|_{c,\chi}^2\leq \frac{C}{k} \left(\|\ddbar^{k,E} s\|_{c,\chi}^2+\|\ddbar_{c,\chi}^{k,E*} s\|_{c,\chi}^2\right) +\int_{K^{'}}|s|_{\chi}^2dv_M,
\end{equation}
where $(s_{1},s_{2})_{c,\chi}:=\int_{M_{c}}\langle s_{1}(x), s_{2}(x)\rangle_{\chi} dv_{M}(x)$.  $\langle\cdot,\cdot\rangle_{\chi}$ is induced by $h^E$, $h_{\chi}^{L^k}$ and some suitable Hermitian metric $\omega$ on $M$.  $\ddbar_{c,\chi}^{k,E*}$ denotes the Hilbert adjoint of $\ddbar^{k,E}$ with $(\cdot,\cdot)_{c,\chi}$ and $K^{'}\Subset M_{c}$.
Thus we deduce the proof of Corollary \ref{mainc}:

\begin{proof}[Proof of Corollary \ref{mainc}]
 By Theorem \ref{main1} and (\ref{covoe}),
        \begin{equation}
		\dim\mathcal{E}^j(k^{-N_{0}},\square^{k,E}_{c,\chi}) \leq\frac{k^{n}}{n!} \int_{K^{'}(j)}(-1)^j c_1(L,h^L_\chi)^n+o(k^{n}).
	\end{equation}
 We need to show that  $h^L_\chi=h^L$ in $K(j)$ and $K^{'}(j)=K(j)$.
   These have been proved in {\cite[Proof of Theorem 1.3]{LSW}}, we omit them here.
\end{proof}

\section{Optimal fundamental estimate for weakly $1$-complete manifolds}

In this section, we will prove that a sublevel set $M_c$ of a weakly $1$-complete manifold $M$, with a Griffiths $q$-positive line bundle $(L,h^L)$ outside a compact set $K$ and a Hermitian line bundle $E$ satisfies the optimal fundamental estimate under a suitable Hermitian metric of $M$.

Consider a domain $M_{i} = \{x \in M:\, \rho(x)<0\} \Subset M$, where $\rho:M \rightarrow \mathbb{R}$ is the defining function that satisfies $|d\rho| = 1$. Let $e_N$ be the inward-pointing unit normal vector field along $ bM_i$, with component $e_N^{(0,1)}$ in $T^{0,1}M$. We have $e_N^{(0,1)} = -\sum_{k=1}^{n} w_k(\rho) \ol{w}_k$, for a local orthonormal basis $\{w_k\}_{k=1}^n$ of $T^{1,0}M$. Set $B^{0, q}(M_i , E) := \{s \in \Omega^{0,q}(\ol{M_i}, E) : i_{e_N^{(0,1)}} s =0 \;\text{on}\;  bM_i \}$, then $B^{0, q}(M_i , E) = \operatorname{Dom}(\ddbar^{E*}) \cap \Omega^{0,q}(\ol{M_i}, E)$.  In $B^{0, q}(M_i , E)$, the Hilbert space adjoint $\ddbar^{E*}_i$ of $\ddbar^{E}$ on $M_i$ agrees with the formal adjoint $\ddbar^{E*}$ of $\ddbar^{E}$  {\cite[Proposition 1.4.19]{MM07}}. Note that $\Omega^{0,q}(\ol{M_i}, E)$ is dense in $\operatorname{Dom} ( \ddbar^E )$ in the graph norm of $\ddbar^E$; $B^{0, q}(M_i , E)$ is dense in $\operatorname{Dom} ( \ddbar^{E*} )$ and in $\operatorname{Dom} ( \ddbar^E ) \cap \operatorname{Dom} ( \ddbar^{E*} )$ in the graph norms of $\ddbar^{E*}$ and $\ddbar^E + \ddbar^{E*}$, respectively {\cite[Proposition 3.5.1]{MM07}}. Let operator $\square_N := \ddbar^E\ddbar^{E*} + \ddbar^{E*}\ddbar^E$ act on $\operatorname{Dom}(\square_N) := \{s \in B^{0, q}(M_i , E): \ddbar^E s \in B^{0, q+1}(M_i , E)\}$. Its Friedrichs extension agrees with the Gaffney extension of Kodaira Laplacian, and is said to be the Kodaira Laplacian with $\ddbar$-Neumann boundary conditions {\cite[Proposition 3.5.2]{MM07}}.

From now on, we suppose that $M$ is a weakly $1$-complete manifold of dimension $n$ and $\varphi$ is the exhaustion function of $M$. Let $(L,h^L)$ and $(E,h^E)$ be holomorphic Hermitian line bundles on $M$ and $K$ be a compact set of $M$. Assume $(L,h^L)$ is Griffiths $q$-positive on $M\setminus K$ with $q\ge 1$.

By Sard's theorem, one can choose an increasing sequence of regular values $\{c_i\}$ of $\varphi$, such that $c_{i}\to \infty$ and for every $c_i$, we have $K\subset M_{i}:=\{\varphi<c_i\}$. For each  $M_i$, since every $c_i$ is regular value and $d\varphi\ne 0$ near $bM_i$, we can use cut-off functions to find a defining function $\rho_{i} := \frac{\varphi-c_i}{|d\varphi|}$ near $bM_i$ {\cite[Page 40]{MM07}}. We have $M_i=\{\rho_{i}<0\}$ and $|d\rho_{i}|_{bM_i}=1$. Because  $\varphi$ is plurisubharmonic function, the Levi form of $\rho_{i}$ is the $2$-form $\cL_{\rho_{i}} :=\dbar\ddbar\rho_{i} \ge 0 \in C^\infty(bM_{i}, T^{(1,0)*}bM_{i} \wedge T^{(0,1)*}bM_{i})$, where $T^{(1,0)}bM_{i}:=\{ v\in T^{(1,0)}M: \dbar\rho_{i}(v)=0 \}$ is the analytic tangent bundle to $bM_i$. Hence, each $M_i$ is smooth pseudoconvex domain. 

Next, let's fix some $M_i$ and assume $K\subset M_{a}\subset M_{i}\subset M_{b}\Subset M$. We can prove

\begin{lemma} \label{lowbd_rho_lem}
	Let $(L,h^L)$ be a Hermitian line bundle on complex manifold $M$, which is Griffiths $q$-positive on the outside of the compact set $K$. For any $C_L>0$ there is a metric $g^{TM}$ (with Hermitian form $\omega)$ on $M$ such that for any $j\geq q$ and any holomorphic Hermitian vector bundle $(E,h^E)$ on $M$,
	\begin{equation}
		 (\Theta_{L}(w_l,\ol{w}_k)\ol{w}^k\wedge i_{\ol{w}_l}s,s)_k\geq C_L||s||^2, \quad s\in \Omega^{0,j}_0(M_b\setminus \overline{M}_a, E),
	\end{equation}
	where $\{ w_l \}_{l=1}^n$ is a local orthonormal frame of $T^{(1,0)}M$ with dual frame $\{ w^l\}_{l=1}^n$ of $T^{(1,0)*}M$.
\end{lemma}

\begin{proof}
	For any $x\in M_b\setminus \overline{M_a}$, let's take a local coordinate $(U,z_{1},\dots,z_{n})$ of $M$ centered at $x$, such that $\Theta_L(\frac{\partial}{\partial z_k},\frac{\partial}{\partial \overline{z}_l})|_x = 0$ if $k\neq l$, by taking a linear transformation of the local basis. Since $L$ is Griffiths $q$-positive, we can also assume that $\Theta_{x,L}$ is positive definite on the subspace of $T_{x}M$ generated by $(\frac{\dbar}{\dbar z_{q}}|_{x},\dots ,\frac{\dbar}{\dbar z_{n}}|_{x})$. For $l\in \mathbb{N}^{+}$ take the metric $g^{TU}_l:=\alpha_{i}^{l} dz_{i}\otimes d\bar{z_i}$, where 
	$\alpha_{i}^{l}:=\left\{
	\begin{aligned}
		&l^{-1}, \quad i\ge q,\\
		&l\quad,  \quad i<q.
	\end{aligned}\right.$
	Set $w_k(x) := (\alpha_{i}^{l})^{-1/2}\frac{\dbar}{\dbar z_{j}}|_{x}$, then $\{w_k(x)\}_{j=1}^{n}$ is an orthonormal basis of $T^{(1,0)}_{x}M$ with respect to metric $g_{l}^{TU}$.

	Denote by $P^{l}_{L}(x)$ the matrix of the linear endomorphism $Q^{l}_{L}(x)$ defined in \eqref{Ql} under the basis $(\frac{\dbar}{\dbar z_{1}}|_{x},\dots ,\frac{\dbar}{\dbar z_{n}}|_{x})$. Since $\Theta_L(\frac{\partial}{\partial z_k},\frac{\partial}{\partial \overline{z}_l}) = 0$ when $k\neq l$, matrix $P^{l}_{L}(x)$ is as
		\begin{equation}\label{Pl}
		P^{l}_{L}(x)=
		\begin{pmatrix}
			\frac{\Theta_{L}\left(\frac{\dbar}{\dbar z_{1}},\frac{\dbar}{\dbar \overline{z}_{1}}\right)}{l}\Big|_x &\ &\ &\ &\ &\  \\
                \  &\ddots &\ &\ &\ &\ \\
                \  &\ & \frac{\Theta_{L}\left(\frac{\dbar}{\dbar z_{q-1}},\frac{\dbar}{\dbar \overline{z}_{q-1}}\right)}{l}\Big|_x &\ &\ &\ \\
                \  &\ &\ & l\cdot\Theta_{L}(\frac{\dbar}{\dbar z_{q}},\frac{\dbar}{\dbar \overline{z}_{q}})|_x &\ &\ \\
                \  &\ &\ &\ &\ddots &\ \\
                \  &\ &\ &\ &\ &  l\cdot\Theta_{L}(\frac{\dbar}{\dbar z_{n}},\frac{\dbar}{\dbar \overline{z}_{n}})|_x
                
		\end{pmatrix}
	\end{equation}
	This is a diagonal matrix with diagonal elements being the eigenvalues of $\Theta_{x,L}$ under the metric $g^{TU}_l$.
	Suppose under the metric $g_{1}^{TU}$ the eigenvalues of $\Theta_{x,L}$ are $d_{1}\le\ d_{2}\le \cdots \le d_{n}$, and $d_{q}>0$. From (\ref{Pl}), we get the eigenvalues of $\Theta_{x,L}$ under the metric $g_{l}^{TU}$ are $\frac{1}{l}d_{1}=:d^{l}_{1}\le\cdots\le\frac{1}{l}d_{q-1}=:d_{q-1}^{l}\le ld_{q}=:d_{q}^{l}\le\cdots \le ld_{n}=:d_{n}^{l}$. Hence, for any given $C_L>0$, there exist sufficient large $l$ and $C$, such that $d_1^l>-C^{-1}$, $d_q^l>C$ and $C-(q-1)C^{-1}>C_{L}+1$.
	Let $\{w^j\}_{j=1}^{n}$ be a local orthonormal basis of $TU$ with respect to metric $g_{l}^{TU}$, and $w^{k}(x)=(\alpha_{k}^{l})^{-1/2}\frac{\dbar}{\dbar z_{k}}|_x, \forall 1\le k\le n$ be the eigenvectors of $Q^{l}_{L}(x)$. Then it follows that for any $s\in\Omega^{0,j}_0(U, E)$, 
	\begin{equation}\label{male}
		\begin{split}
			&\langle \Theta_{x,L}(w_l(x),\ol{w}_k(x))\ol{w}^k(x)\wedge i_{\ol{w}_l(x)}s(x),s(x) \rangle_{g_{l}^{TU}}\\
			=&\sum_{i=1}^{n}\langle d_{i}^{l}\ol{w}^i(x)\wedge i_{\ol{w}_i(x)}s(x),s(x) \rangle_{g_{l}^{TU}}\\
			=&\sum_{i_{\ol{w}_i(x)}s(x)\ne 0} d_{i}^{l}\langle s(x),s(x) \rangle_{g_{l}^{TU}}
			\ge \sum_{i=1}^{j} d_{i}^{l}|s(x)|^{2}_{g_{l}^{TU}}\\
			\ge &((j-q+1)C-(q-1)C^{-1})|s(x)|^{2}_{g_{l}^{TU}}\\
			\ge &(C-(q-1)C^{-1})|s(x)|^{2}_{g_{l}^{TU}}\\
                \ge &(C_{L}+1)|s(x)|^{2}_{g_{l}^{TU}}.
		\end{split}
	\end{equation}
 The forth inequality comes from $j\ge q$. By the continuity of the inequality \eqref{male}, there is a neighborhood $U^{'} \subset U$ of $x$, such that under the metric $g_{l}^{TU}$,
 \begin{equation}\label{local ineq}
     \langle \Theta_{x,L}(w_l(x),\ol{w}_k(x))\ol{w}^k(x)\wedge i_{\ol{w}_l(x)}s(x),s(x) \rangle_{g_{l}^{TU}} \ge C_{L}|s(x)|^{2}_{g_{l}^{TU}},\quad \forall x\in U^{'}, s\in\Omega^{0,j}_0(U, E),
 \end{equation}
 where $\set{w_i}_{i=1}^{n}$ is an orthonormal basis with respect to $g_{l}^{TU}$.
Summarizing the above, we get for any $x\in M$, there is a  pair $(V(x),g^{TV(x)}_{x})$, where $V(x)$ is a neighborhood of $x$ and  $g^{TV(x)}_{x}$ is a metric on $TV(x)$. Obviously, $\set{V(x)}_{x\in \overline{M_b\setminus \overline{M}_a}}$ is an open cover of relatively compact set $ M_b\setminus \overline{M}_a$, thus we can choose a finite subcover of it. Let $\set{V_i}_{i=1}^{N}$ be the subcover and $g^{TV_{i}}_{i}$ be the metric correspond to $V_i$.  There is a partition of unity $\set{\rho_i}_{i=1}^{N}$ subordinate to $\set{V_i}_{i=1}^{N}$. Define $g^{TM}:=\sum_{i=1}^{N}\rho_{i}g^{TV_{i}}_{i}$, suppose $\set{w_k}_{k=1}^{n}$ is a local orthonormal basis with respect to $g^{TM}$ and $\set{v_{k(i)}}_{k=1}^{n}$ is a local orthonormal basis with respect to  $g^{TV_{i}}_{i}$. Assume $w_{k}=a_{k}^{j(i)}v_{j(i)}, w^{k}=b^{k}_{j(i)}v^{j(i)}$, thus $a_{k}^{j(i)}b^{l}_{j(i)}=\delta_{k}^{l}$.
Under the metric $g^{TM}$ for any $x\in M_b\setminus \overline{M}_a$ and $s\in\Omega^{0,j}_0(M_b\setminus \overline{M}_a, E)$,
\begin{equation}
\begin{split}
     &\langle \Theta_{x,L}(w_l(x),\ol{w}_k(x))\ol{w}^k(x)\wedge i_{\ol{w}_l(x)}s(x),s(x) \rangle_{g^{TM}} \\
    =&-\langle \Theta_{x,L}(w_l(x),\ol{w}_k(x))\ol{w}^k(x)\wedge s(x),\ol{w}^l(x)\wedge s(x) \rangle_{g^{TM}}\\
    =&-\sum_{i=1}^{N}\rho_i(x)\langle \Theta_{x,L}(w_l(x),\ol{w}_k(x))\ol{w}^k(x)\wedge s(x),\ol{w}^l(x)\wedge s(x) \rangle_{g^{TV_{i}}_{i}}\\
    =&-\sum_{i=1}^{N}\rho_i(x) \langle \Theta_{x,L}(a_{l}^{\alpha (i)}v_{\alpha (i)}(x),\ol{a}_{k}^{\beta(i)}\ol{v}_{\beta(i)}(x))\ol{b}^{k}_{\gamma(i)}\ol{v}^{\gamma(i)}(x)\wedge s(x),\ol{b}^{l}_{\lambda(i)}\ol{v}^{\lambda(i)}(x)\wedge s(x) \rangle_{g^{TV_{i}}_{i}}\\
    =&-\sum_{i=1}^{N}\rho_i(x) a_{l}^{\alpha (i)}\ol{a}_{k}^{\beta(i)}\ol{b}^{k}_{\gamma(i)}b^{l}_{\lambda(i)}\langle \Theta_{x,L}(v_{\alpha (i)}(x),\ol{v}_{\beta(i)}(x))\ol{v}^{\gamma(i)}(x)\wedge s(x),\ol{v}^{\lambda(i)}(x)\wedge s(x) \rangle_{g^{TV_{i}}_{i}}\\
    =&-\sum_{i=1}^{N}\rho_i(x)\langle \Theta_{x,L}(v_{\alpha (i)}(x),\ol{v}_{\beta(i)}(x))\ol{v}^{\beta(i)}(x)\wedge s(x),\ol{v}^{\alpha(i)}(x)\wedge s(x) \rangle_{g^{TV_{i}}_{i}}\\
    =&\sum_{i=1}^{N}\rho_i(x)\langle \Theta_{x,L}(v_{\alpha (i)}(x),\ol{v}_{\beta(i)}(x))\ol{v}^{\beta(i)}(x)\wedge i_{\ol{v}_{\alpha(i)}(x)}s(x), s(x) \rangle_{g^{TV_{i}}_{i}}\\
    \ge& \sum_{i=1}^{N}\rho_i(x)C_{L}|s(x)|^{2}_{g^{TV_{i}}_{i}}=C_{L}|s(x)|^{2}_{g^{TM}}.
\end{split}
\end{equation}
 The second to last inequality comes from \eqref{local ineq}. Finally, we can  finish the proof by taking integral of $\langle \Theta_{x,L}(w_l(x),\ol{w}_k(x))\ol{w}^k(x)\wedge i_{\ol{w}_l(x)}s(x),s(x) \rangle_{g^{TM}}$ on $M_b\setminus \overline{M}_a$.
 
\end{proof}

We should note that the proof is inspired by {\cite[Lemma 3.5.3]{MM07}}. From now on, we will assign the Hermitian metric $\omega$ in (\ref{lowbd_rho_lem}) to the given $1$-weakly complete manifold $M$. It's well known that the Hilbert space adjoint $\ddbar^{E*}_i$ on $M_i$ agrees with the formal adjoint $\ddbar^{E*}$ on $B^{0, j}(M_i , E)$, for any relatively compact domain $M_i$ in $M$ and any holomorphic vector bundle $E$,  $1\leq j \leq n$.

\begin{lemma}\label{keylem}
	Let $M$ be a weakly $1$-complete manifold, $(E,h^E)$ be a Hermitian line bundle and $(L,h^L)$ be a Griffiths $q$-positive line bundle on the outside of a compact set $K$, we deduce  the inequality below, 
	\begin{equation}
		||s||_i^2\leq \frac{C_1}{k}( ||\ddbar^{k,E} s||_i^2+||\ddbar^{k,E*} s||_i^2 )
	\end{equation}
	for $s\in B^{0,j}(M_i,L^k\otimes E)$ with $\supp(s)\subset M_b\setminus \overline{M}_a$, $j\geq q$ and $k\gg 1$,
	where $C_1>0$ and the $L^2$-norm $||\cdot||_i$ is given by $\omega$, $h^{L^k}$ and $h^E$ on $M_i$.
\end{lemma}
\begin{proof}
	Let $\rho_{i}$ be the defining function of $M_i$. Using Theorem \ref{BKNwithbd}, for any $s \in B^{0,j}(M_i,L^k\otimes E),0\leq j\leq n$, we deduce that 
	\begin{equation}\label{eq_BKN}
		\begin{split}
			&\frac{3}{2}(||\ddbar^{k,E} s||_i^2+||\ddbar^{k,E*} s||_i^2)\\
			\geq &(  \Theta_{L^k\otimes E\otimes K^*_X}(w_j,\ov{w}_k)\ov{w}^k\wedge i_{\ov{w}_j} s,s)_i
			+\int_{bM_i}\cL_{\rho_i}(s,s)dv_{bM_i}
			-\frac{1}{2}(||\mT^*s||_i^2+||\ov{\mT}\til{s}||_i^2+||\ov{\mT}^* \til{s}||_i^2)\\
			=& k( \Theta_{L}(w_j,\ov{w}_k)\ov{w}^k\wedge i_{\ov{w}_j} s,s)_i
			+( \Theta_{E\otimes K^*_X}(w_j,\ov{w}_k)\ov{w}^k\wedge i_{\ov{w}_j} s,s)_i
			-\frac{1}{2}(||\mT^*s||_i^2+||\ov{\mT}\til{s}||_i^2+||\ov{\mT}^* \til{s}||_i^2)\\
			+&\int_{bM_i}\cL_{\rho_i}(s,s)dv_{bM_i}.
		\end{split}
	\end{equation}
	Since $\varphi-c_{i}=i_{e_N^{(0,1)}} s =0$ on $bM_i$ and $\varphi$ is a plurisubharmonic function on $M$, we have	
	\begin{equation}\label{Levi}
		\,\,\,\,\,\begin{split}
			&\int_{bM_i}\cL_{\rho_i}(s,s)dv_{bM_i}\\
			=&\int_{bM_i}\langle \dbar\ddbar \rho_{i}(w_j,\ov{w}_k)\ov{w}^k\wedge i_{\ov{w}_j} s,s\rangle dv_{bM_i}
			=\int_{bM_i}\left\langle \dbar\ddbar\left(\frac{\varphi-c_i}{|d\varphi|}\right)(w_j,\ov{w}_k)\ov{w}^k\wedge i_{\ov{w}_j} s,s\right\rangle dv_{bM_i}\\
			=&\int_{bM_i}\frac{1}{|d\varphi|}\langle \dbar\ddbar\varphi(w_j,\ov{w}_k)\ov{w}^k\wedge i_{\ov{w}_j}s,s\rangle dv_{bM_i}
			+\int_{bM_i}\left\langle \dbar\frac{1}{|d\varphi|}\wedge\ddbar\varphi(w_j,\ov{w}_k)\ov{w}^k\wedge i_{\ov{w}_j}s,s\right\rangle dv_{bM_i}\\
			+&\int_{bM_i}\left\langle \dbar\varphi\wedge\ddbar\frac{1}{|d\varphi|}(w_j,\ov{w}_k)\ov{w}^k\wedge i_{\ov{w}_j}s,s\right\rangle dv_{bM_i}
			+\int_{bM_i}(\varphi-c_i)\left\langle \dbar\ddbar\frac{1}{|d\varphi|}(w_j,\ov{w}_k)\ov{w}^k\wedge i_{\ov{w}_j}s,s\right\rangle dv_{bM_i}\\
			=&\int_{bM_i}\frac{1}{|d\varphi|}\langle \dbar\ddbar\varphi(w_j,\ov{w}_k)\ov{w}^k\wedge i_{\ov{w}_j}s,s\rangle dv_{bM_i}
			\ge 0.
		\end{split}
	\end{equation}
	
	Notice that $M_i$ is a relatively compact set of $M$ and $\mT^*, \ov{\mT}, \ov{\mT}^*, \Theta_{E\otimes K^*_X}$ are smooth operators on $M$, hence there is a $C^{'}\ge 0$ such that,
	\begin{equation}\label{tau}
		(\Theta_{E\otimes K^*_X}(w_j,\ov{w}_k)\ov{w}^k\wedge i_{\ov{w}_j} s,s)_i
		-\frac{1}{2}(||\mT^*s||_i^2+||\ov{\mT}\til{s}||_i^2+||\ov{\mT}^* \til{s}||_i^2
		\ge -C^{'}||s||_i^2 .
	\end{equation}
	
	Using (\ref{eq_BKN}), (\ref{Levi}), (\ref{tau}) and combining Lemma\ref{lowbd_rho_lem}, we get
	\begin{equation}\label{eq_posline}	
		\frac{3}{2}(||\ddbar^{k,E} s||_i^2+||\ddbar^{k,E*} s||_i^2)
		\ge (kC_{L}-C^{'})||s||_i^2 .
	\end{equation} 	
	
	Take $k_{0}:=2[\frac{C^{'}}{C_L}]+1$, then for any $k\ge k_0$, we have $kC_L -C^{'}>\frac{k}{2}C_L$, which means
	\begin{equation}
		\frac{3}{2}(||\ddbar^{k,E} s||_i^2+||\ddbar^{k,E*} s||_i^2)
		\ge (kC_{L}-C^{'})||s||_i^2 \ge \frac{k}{2}C_L||s||_i^2.
	\end{equation}
	Let $C_1 :=\frac{3}{C_L}$, then we get
	\begin{equation}
		||s||_i^2\le \frac{3}{kC_L}(||\ddbar^{k,E} s||_i^2+||\ddbar^{k,E*} s||_i^2)=\frac{C_1}{k}(||\ddbar^{k,E} s||_i^2+||\ddbar^{k,E*} s||_i^2).
	\end{equation}
\end{proof}

\begin{remark}
	In {\cite[Lemma3.3]{LSW}}, for $q$-convex manifolds, to get the same Lemma as (\ref{keylem}), they must modify the Hermitian metric $h^L$ of line bundle $L$ by using some rapidly increasing-convex function $\chi$. But we don't need to do this, since the positive property in our case is at the curvature of line bundle $L$.
\end{remark}

Finally, we are in the position to prove the optimal fundamental estimate for $M_i$ with $L$.

\begin{proposition}\label{1coxFE}
	Let $M$ be a weakly $1$-complete  manifold of dimension $n$, $(L,h^L)$ be a Hermitian line bundle on complex manifold $M$, which is Griffiths $q$-positive on the outside of the compact set $K\subset M_i$ and $(E,h^E)$ be a Hermitian line bundle. Then there exist a compact subset $K'\subset M_i$ with $K\subset K'$ and $C>0$ such that for sufficiently large $k$, we deduce that
	\begin{equation}
		(1-\frac{C}{k})||s||_i^2\leq \frac{C}{k}(||\ddbar^{k,E}s||_i^2+||\ddbar^{k,E*}_{i}s||_i^2)+\int_{K'} |s|^2 dv_M
	\end{equation}
	for any $s\in \Dom(\ddbar^{k,E})\cap \Dom(\ddbar^{k,E*}_{i})\cap L^2_{0,j}(M_{i},L^k\otimes E)$ and $q\leq j \leq n$,
	where the $L^2$-norm is given by $\omega$, $h^{L^k}$ and $h^E$ on $M_i$.	
	
\end{proposition}
\begin{proof}
	First, we suppose  $s\in B^{0,j}(M_i,L^k\otimes E)$ with $j\ge q$ .
	
	Assume $a<a^{'}<c_i<b^{'}<b\in\mathbb{R}$. Let $\zeta\in C^\infty_0(M_b,\R)$ with $\supp(\zeta)\subset M_b\setminus \ov{M}_a$ such that $0\leq \zeta \leq 1$ and $\zeta=1$ on $M_{b^{'}}\setminus \ov{M}_{a^{'}}$. We define $K':=\ov{M}_{a^{'}}$. Then we deduce that
	\begin{equation}\label{K'lem}
		\begin{split}
			&||\zeta s||_i^2+\int_{K'}|s|^2 dv_M\\
			=&\int_{M_i}|\zeta s|^2dv_M+\int_{K'}|s|^2 dv_M
			=\int_{M_i\setminus \ov{M}_a}|\zeta s|^2dv_M+\int_{\ov{M}_{a^{'}}}|s|^2 dv_M\\
			=&\int_{M_i\setminus \ov{M}_{a^{'}}}|\zeta s|^2 dv_M
			+\int_{\ov{M}_{a^{'}}\setminus \ov{M}_a }|\zeta s|^2 dv_M
			+\int_{\ov{M}_{a^{'}}}|s|^2 dv_M\\
			=&\int_{M_i\setminus \ov{M}_{a^{'}}}|s|^2 dv_M
			+\int_{\ov{M}_{a^{'}}}|s|^2 dv_M
			+\int_{\ov{M}_{a^{'}}\setminus \ov{M}_a }|\zeta s|^2 dv_M\\
			=&||s||_i^2+\int_{\ov{M}_{a^{'}}\setminus \ov{M}_a }|\zeta s|^2 dv_M
			\ge ||s||_i^2.
		\end{split}
	\end{equation}

	We know $\zeta s\in \Omega^{0,j}(\ov{M}_i,L^k\otimes E)$, and $i_{e_{\bn}^{(0,1)}}(\phi s)=i_{e_{\bn}^{(0,1)}}( s)=0$ on $bM_i$, since $\zeta s$  coincides with $s$ in $M_{b^{'}}\setminus \ov{M}_{a^{'}}$, which is a neighborhood of $bM_i$. Thus $\zeta s \in B^{0,j}(M_i,L^k\otimes E)$ with $\supp(\zeta s)\subset M_b\setminus \ov{M}_a$. By Lemma \ref{keylem}, there exists $C_1>0$ such that for $j\geq q$ and $k\gg 1$, we have
	\begin{equation}\label{eq_mid}
		||\zeta s||_i^2\leq \frac{C_1}{k}( ||\ddbar^{k,E} (\zeta s)||_i^2+||\ddbar^{k,E*} (\zeta s)||_i^2 ).
	\end{equation}
	Let $\eta:=1-\zeta$ and $C^{''}:=\sup_{x\in M_i}|d\eta(x)|_{g^{T^*M}}^2>0$. Then, for any $s\in B^{0,j}(M_{i},L^k\otimes E)$ and $k\geq 1$, we now claim that
	\begin{equation}\label{cuteq_k}
		\frac{1}{k}(||\ddbar^{k,E} (\zeta s)||_i^2+||\ddbar^{k,E*}(\zeta s)||_i^2)
		\leq \frac{5}{k}(||\ddbar^{k,E} s||_i^2+||\ddbar^{k,E*} s||_i^2)+\frac{12C^{''}}{k}||s||_i^2.
	\end{equation}
	Notice that $0\leq \eta\leq 1$, $\eta=1$ on $\ov M_{a}$ and $\eta=0$ on $\ov M_i\setminus \ov{M}_{a^{'}}$, thus $\eta \in C^\infty_0(M_i)$. So $\eta s\in \Omega^{0,j}_0(M_i,L^k\otimes E)\subset B^{0,j}(M_i,L^k\otimes E)$.
	For simplifying notations, we use $\ddbar$ and $\ddbar^*$ instead of $\ddbar^{k,E}$ and $\ddbar^{k,E*}$ respectively. 
	It follows that $||\ddbar (\eta s)||^2=||\ddbar\eta\wedge s+\eta \ddbar s||^2\quad\mbox{and} \quad
	||\ddbar^*(\eta s)||^2=||\eta\ddbar^*s-i_{\ddbar\eta}s||^2$.
	And notice that $||\ddbar\eta\wedge s||^2\leq C^{''}||s||^2$ and $||i_{\ddbar\eta}s||^2\leq C^{''}||s||^2$. 
	By Cauchy's inequality, we get 
	\begin{equation}
		\begin{split}
			&||\ddbar (\zeta s)||_i^2+||\ddbar^*(\zeta s)||_i^2\\
			\le &2(||\ddbar s-\ddbar(\zeta s)||_i^2+||\ddbar^* s-\ddbar^*(\zeta s)||_i^2)+2(||\ddbar s||_i^2+||\ddbar^* s||_i^2)\\
			=&2(||\ddbar (\eta s)||_i^2+||\ddbar^* (\eta s)||_i^2)+2(||\ddbar s||_i^2+||\ddbar^* s||_i^2)\\
			=&2(||\ddbar\eta\wedge s+\eta \ddbar s||_i^2+||\eta\ddbar^*s-i_{\ddbar\eta}s||_i^2)+2(||\ddbar s||_i^2+||\ddbar^* s||_i^2)\\
			\le & 3(||\eta \ddbar s||_i^2+||\eta\ddbar^*s||_i^2)+6(||\ddbar\eta\wedge s||_i^2+||i_{\ddbar\eta}s||_i^2)+2(||\ddbar s||_i^2+||\ddbar^* s||_i^2)\\
			\le & 3(|| \ddbar s||_i^2+||\ddbar^*s||_i^2)+6(C^{''}||s||_i^2+C^{''}||s||_i^2)+2(||\ddbar s||_i^2+||\ddbar^* s||_i^2)\\
			=& 5(|| \ddbar s||_i^2+||\ddbar^*s||_i^2)+12C^{''}||s||_i^2.
		\end{split}
	\end{equation}
	This finishes the proof of \eqref{cuteq_k}.

	Next by applying \eqref{K'lem}, \eqref{eq_mid} and \eqref{cuteq_k}, we obtain
	\begin{equation}
		||s||_i^2-\int_{K'}|s|^2dv_M\leq \frac{5C_1}{k}(||\ddbar^{k,E} s||_i^2+||\ddbar^{k,E*} s||_i^2)+\frac{12C_1C^{''}}{k}||s||_i^2.
	\end{equation}
	Since $B^{0,j}(M_i,L^k\otimes E)$ is dense in $\Dom(\ddbar^{k,E})\cap \Dom(\ddbar^{k,E*}_{i})\cap L^2_{0,j}(M_{i},L^k\otimes E)$ with respect to the graph norm of $\ddbar^{k,E}+\ddbar^{k,E*}_{i}$ and $\ddbar^{k,E*}_{i}$  coincides with $\ddbar^{k,E*}$ on $B^{0,j}(M_i,L^k\otimes E)$, we can complete this proof by taking $C:=\max \{5C_1, 12C_1C^{''}\}.$
	
\end{proof}

This proof is inspired by {\cite [Proposition 3.4]{LSW}}. Using Proposition {\ref{1coxFE}} and Theorem {\ref{main1}}, in the end, we can prove the main theorem.

\begin{proof}[Proof of Theorem \ref{maint}]
	By the choice of $M_i$, we only need to prove for any $i$ the inequality holds. By Proposition {\ref{1coxFE}} and Theorem {\ref{main1}}, we deduce that
	\begin{equation}
		\dim\mathcal{E}^j(k^{-N_{0}},\square^{k,E}_{i}) \leq\frac{k^{n}}{n!} \int_{K^{'}(j)}(-1)^j c_1(L,h^L)^n+o(k^{n}).
	\end{equation}
	It only needs to prove that $K(j)=K^{'}(j)$ for $j\ge q$. Since $K\subset K^{'}$, we have $K(j)\subset K^{'}(j)$. For the opposite direction, since $L$ is Griffiths $q$- positive on $M\setminus K$,  $\Theta_{L}$ has $n-q+1$ positive eigenvalues in $M\setminus K$ at least. So $\Theta_{L}$ has $n-(n-q+1)=q-1<j$ negative eigenvalues at most. This completes this proof. 
\end{proof}

\begin{corollary}
	Let $M$ be a weakly $1$-complete manifold of dimension $n$. Let $(L,h^L)$ and $(E,h^E)$ be  holomorphic Hermitian line bundles on $M$. Assume $K$ is a compact subset and $(L,h^L)$ is Griffiths $q$-positive on $X\setminus K$ with $q\geq 1$. Then, there exists a sequence of relatively compact sets $\{M_i\}_{i=1}^{\infty}$ of $M$, such that $K\subset M_1$, $\cup_{i} M_i=M$ and $M_i\Subset M_{i+1}$ , and a sequence of Hermitian metrics $\{\omega_i\}_{i=1}^{\infty}$ on $M$. As $k\rightarrow \infty$, 
    for any $j\geq q$ and $i\in \mathbb{N}^+$, we have
	\begin{equation}
		\dim H^{0,j}_{(2)}(M_i,L^{k}\otimes E) \leq\frac{k^{n}}{n!} \int_{K(j)}(-1)^j c_1(L,h^L)^n+o(k^{n}).
	\end{equation}
	In particular, if $L>0$ on $X\setminus K$, we deduce that 
	\begin{equation}
		\dim H^{j}(M,L^{k}\otimes E) \leq\frac{k^{n}}{n!} \int_{K(j)}(-1)^j c_1(L,h^L)^n+o(k^{n}),
	\end{equation}
	for every $j\ge1$.
\end{corollary}
\begin{proof}
	The first statement is from Theorem{\ref{maint}} and $\dim H^{0,j}_{(2)}(M_i,L^{k}\otimes E)\le \dim\mathcal{E}^j(k^{-N_{0}},\square^{k,E}_{i})$. When $L>0$, by\cite[Theorem 6.2]{Tak83}, we have $ H^{j}(M,L^{k}\otimes E)\cong \cH^{0,j}(M_i,L^{k}\otimes E)\cong  H^{0,j}_{(2)}(M_i,L^{k}\otimes E)$ for $k\gg1$ and every $i\in \mathbb{N}^+$, thus we get the second statement.
\end{proof}

\bibliographystyle{amsalpha}

\end{document}